\documentclass[ejs,preprint]{imsart}

\RequirePackage[OT1]{fontenc}
\RequirePackage{amsthm,amsmath,amssymb}
\RequirePackage[numbers]{natbib}
\RequirePackage[colorlinks,citecolor=blue,urlcolor=blue]{hyperref}
\usepackage{amstext} 
\usepackage{array}   
\usepackage{enumitem} 
\setlist[itemize]{noitemsep} 
\allowdisplaybreaks
\usepackage{multirow}
\usepackage[flushleft]{threeparttable}
\usepackage{verbatim}
\usepackage{hyperref}
\usepackage{bm,bbm}
\usepackage{longtable}
\usepackage{rotating}
\usepackage{caption}

\usepackage{indentfirst}

\pubyear{2020}
\volume{0}
\issue{0}
\firstpage{1}
\lastpage{8}

\startlocaldefs
\numberwithin{equation}{section}
\theoremstyle{plain}
\newtheorem{theorem}{Theorem}[section]
\newtheorem{lemma}{Lemma}[section]
\endlocaldefs

\newcommand{\N}{\mathrm{N}}
\renewcommand{\P}{\mathrm{P}}
\newcommand{\E}{\mathrm{E}}

\newcommand{\LL}{\mathbb{L}}

\newcommand{\trans}{^{\tiny{\mathrm{T}}}}
\newcommand{\Ind}{\mathbbm{1}}

\begin{document} 
	
\begin{frontmatter}
	\title{Convergence Rates for Bayesian Estimation and Testing in Monotone Regression}
	\runtitle{Bayesian Monotone Regression}
	
	\begin{aug}
		\author{\fnms{Moumita} \snm{Chakraborty}\ead[label=e1]{mchakra@ncsu.edu}}
		\address{Department of Operations Research \\
			North Carolina State University\\
			Raleigh, NC 27695\\
			U.S.A.\\
			\printead{e1}}
		\and
		\author{\fnms{Subhashis} \snm{Ghosal}\thanksref{t1}\ead[label=e2]{sghosal@stat.ncsu.edu}}
		\address{Department of Statistics\\
			North Carolina State University\\
			Raleigh, NC 27695\\
			U.S.A.\\
			\printead{e2}}
	
		\runauthor{Chakraborty and Ghosal}
		\thankstext{t1}{Research is partially supported by NSF grant number DMS-1916419.}
	\end{aug}

\begin{abstract}
  Shape restrictions such as monotonicity on functions often arise naturally in statistical modeling. 
  We consider a Bayesian approach to the problem of estimation of a monotone regression function and testing for monotonicity. We construct a prior distribution using piecewise constant functions. For estimation, a prior imposing monotonicity of the heights of these steps is sensible, but the resulting posterior is harder to analyze theoretically. We consider a ``projection-posterior'' approach, where a conjugate normal prior is used, but the monotonicity constraint is imposed on posterior samples by a projection map on the space of monotone functions. We show that the resulting posterior contracts at the optimal rate $n^{-1/3}$ under the $\LL_1$-metric and at a nearly optimal rate under the empirical $\LL_p$-metrics for $0<p\le 2$. The projection-posterior approach is also computationally more convenient. We also  construct a Bayesian test for the hypothesis of monotonicity using the posterior probability of a shrinking neighborhood of the set of monotone functions. We show that the resulting test has a universal consistency property and obtain the separation rate which ensures that the resulting power function approaches one. 
\end{abstract}


\end{frontmatter}

\section{Introduction}
\label{intro}

We consider the nonparametric regression model $Y=f(X)+\varepsilon$ for a response variable $Y$ with respect to a one-dimensional predictor variable $X\in [0, 1]$ (without loss of generality) and $\varepsilon$ a mean-zero random error with finite variance $\sigma^2$. Instead of the more commonly imposed smoothness condition, 
$f$ is assumed to be a monotone increasing function on $[0, 1]$.  We observe $n$ replications $(Y_1, X_1), \ldots,(Y_n, X_n)$, where the design points $X_1, \ldots, X_n$ are either deterministic or are randomly sampled from a fixed distribution $G$. The error $\varepsilon$ is assumed to be distributed independently of the predictor $X$.

The problem has been widely studied in the frequentist literature, and is commonly known as isotonic regression. Barlow and Brunk \cite{BarlowBrunk} obtained the greatest convex minorant (GCM) of a cumulative sum diagram as the least-square estimator under the monotonicity constraint. The Pool-Adjacent-Violators Algorithm (PAVA) describes a method of successive approximation to the GCM, and is the most commonly used algorithm for isotonic regression (see Ayer et al. \cite{ayerPAVA},  Barlow et al. \cite{BarlowStatInf}, or De Leeuw et al. \cite{PAVAinR}). Brunk \cite{Brunk1973} showed that the estimated value of the regression function at a point converges at a rate $n^{-1/3}$, and evaluated its asymptotic distribution. Durot \cite{DUROTsharpAsymp} established $n^{-1/3}$ rate of convergence of the isotonic regression estimator under the $\LL_1$-metric.

A Bayesian approach to the monotone regression problem involves putting a prior on functions under the monotonicity constraint. Since step-functions can approximate monotone functions, a natural approach is to put priors on step heights under the monotonicity constraint, and possibly also on the locations and the number of intervals. For smoother sample paths, higher-order splines can be used instead of the indicator functions of intervals. Shivley \cite{ShivelyEstimation} used a mixture of constrained normal distributions as a prior for spline coefficients. 
Bayesian nonparametric methods have been developed also for other shape-constrained problems, such as monotone density and current status censoring model. Salomond \cite{Salomond2014a} established the nearly minimax rate $n^{-1/3}$ for a decreasing density using a mixture of uniform densities as a prior. 
Testing for monotonicity of a regression function has been studied in the frequentist literature by Akakpo \cite{akakpo2014}, Hall and Heckman \cite{hall2000}, Baraud et al. \cite{baraud2005}, Ghosal et al. \cite{ghosalMonotonicityTesting} and Bowman et al. \cite{bowmanTesting}. A Bayesian approach to testing monotonicity was proposed by Salomond \cite{salomond2018}. 

A difficulty with the usual Bayesian approach to isotonic regression is that the monotonicity constraint on the coefficient makes both posterior computation and study of posterior concentration with increasing sample size a lot more challenging. This is especially the case if the true regression function lies on the boundary of the set of monotone functions, since then the prior puts a relatively less mass in the neighborhood of the true regression function. A very useful approach that can still utilize the conjugacy structure is provided by a ``projection-posterior'' distribution. In this approach, the monotonicity constraint on the step size is initially ignored, so that they may be given independent normal priors, and hence the posterior distribution is also normal, allowing easy sampling, and large sample analysis of posterior concentration. Then a posterior distribution is directly induced by a projection map that projects a step function to the nearest monotone function in terms of the $\LL_1$-distance or some other metric. A similar idea based on a Gaussian process prior was used by Lin and Dunson \cite{LizhenDunson} for monotone regression. Bhaumik and Ghosal \cite{bhaumik2015bayesian,bhaumik2017efficient} used this idea of embedding in an unrestricted space and then projecting a conjugate posterior in regression models driven by ordinary differential equations. In this paper, we pursue the projection-posterior  approach and show that the resulting projection-posterior distribution concentrates at the optimal rate $n^{-1/3}$ in terms of the $\LL_1$-distance. We obtain nearly optimal posterior concentration under an empirical $\LL_p$-distance for $0<p\le 2$. We also construct a Bayesian test for the hypothesis of monotonicity based on the posterior distribution of the difference between the unrestricted posterior sample and its projection. We show that the resulting test is universally consistent, in that the Type I error probability goes to zero and the power goes to one at any fixed alternative, regardless of smoothness. For a sequence of smooth alternatives, we also compute the needed separation from the null region to obtain high power. Our proposed test is similar in spirit to Salomond's \cite{salomond2018} test in that both are based on the posterior probability of a slightly extended null region, but our use of the $\LL_1$-metric on the function or the Hellinger metric on the density of $Y$, leads to the universal consistency.      

The paper is organized as follows. In the next section, we formally introduce the modeling assumptions and the prior and describe the projection posterior approach. In Section~\ref{sec:rates}, we present results on posterior contraction rates of the projection posterior distribution. In Section~\ref{sec:testing}, we derive asymptotic properties of the proposed Bayesian tests. Proofs of the main results are given in Section~\ref{sec:proofs} and those of the auxiliary results in Section~\ref{sec:appendix}. 

\section{Model, prior and projection posterior}
\label{sec:setup}

The following notations will be used throughout the paper. Let $\boldsymbol{I}_m$ stand for the $m\times m$ identity matrix.
By $\boldsymbol{Z} \sim \mathrm{N}_J(\boldsymbol{\mu}, \boldsymbol{\Sigma})$, we mean that $\boldsymbol{Z}$ has a $J$-dimensional normal distribution with mean $\boldsymbol{\mu}$ and covariance matrix $\boldsymbol{\Sigma}$. For a vector $\bm{x}$, the Euclidean norm will be denoted by $\|\bm x\|$. The transpose of a vector $\bm x$ is denoted by $\bm x\trans$ and that of a matrix $\bm A$ is denoted by $\bm A\trans$. If $f$ is a function and $H$ a measure, the $\LL_p$-norm of $f$ is given by $\|f\|_{p,H}=(\int |f|^p dH)^{1/p}$ for $1\le p<\infty$, and the $\LL_p$ distance between two functions $f$ and $g$ is given by $d_{p,H}(f,g)=\|f-g\|_{p,H}$ for $1\le p<\infty$ and $d_{p,H}(f,g)=\int |f-g|^p dH$ for $0<p<1$. The indicator function will be denoted by $\Ind$ and $\#$ will stand for the cardinality of a finite set. 

For two sequences of real numbers $a_n$ and $b_n$, $a_n\lesssim b_n$ means that $a_n/b_n$ is bounded, $a_n \asymp b_n$ means that both $a_n\lesssim b_n$ and $b_n \lesssim a_n$, and $a_n \ll b_n$ means that $a_n/b_n \rightarrow 0$. For a random variable $Y$ and a sequence of random variables $X_n$, $X_n\rightarrow_P Y$ means that $X_n$ converges to $Y$ in $P$-probability. 

Let $\mathcal{F}$ and $\mathcal{F}_+$ respectively denote the space of real-valued measurable functions and monotone increasing functions on $[0, 1]$, and for $K>0$, let $\mathcal{F}_+(K)=\{f\in \mathcal{F}_+: |f|\le K\}$.  
For $f:[0, 1]\mapsto \mathbb{R}$ and $d$ a distance on $\mathcal{F}$, let the projection of $f$ on $\mathcal{F}_+$ be the function $f^*$ that minimizes $d(f, h)$ over $h\in\mathcal{F}_+$. The topological closure of $\mathcal{F}_+$ is denoted by $\bar{\mathcal{F}}_+$. The $\epsilon$-covering number of a set $A$ with respect to a metric $d$, denoted by $\mathcal{N}(\epsilon,A,d)$, is the minimum number of balls of radius $\epsilon$ needed to cover $A$. 

Let $G_n(x)=n^{-1} \sum_{i=1}^n \mathbbm{1}\{X_i\le x \}$, the empirical distribution of the predictors $X$. 

A prior distribution on the regression function $f$ will be given by a random step function $f(x)=\sum_{j=1}^J \theta_j \mathbbm{1} \{ x\in I_j\}$, $x\in (0,1]$  where $I_1,\ldots,I_J$ are disjoint intervals partitioning $[0,1]$ given by $I_j =(\xi_{j-1},\xi_{j}]$, $j=1,\ldots,J-1$, and $I_J=[\xi_{J-1},\xi_J]$. The knot points are $0=\xi_0<\xi_1<\ldots<\xi_{J-1}<\xi_J=1$. With a given set of $J$ knots, the corresponding collection of step functions is denoted by $\mathcal{F}_J$. The counts of these intervals are denoted by $N_j =\sum_{i=1}^n  \mathbbm{1} \{ X_i\in I_j\}$, $j=1,\ldots,J$. For the prior, $J$ or $\bm \xi=(\xi_1,\ldots,\xi_{J-1})$ or both may be given, or these may may be distributed according to a prior.  Depending on their choices, the following three types of prior distributions will be considered in this paper. 

\begin{enumerate}
	\item \textbf{Type 1 prior}: The number of steps $J$ is deterministic (will be dependent on the sample size $n$), $\xi_j=j/J$, $j=1,\ldots,J-1$. 
	\item \textbf{Type 2 prior}: The number of steps  $J$ is deterministic, 
	$$\P ((\xi_1,\ldots,\xi_{J-1})=S)=\frac{1}{\binom{n}{J-1}}, \; S\subset \{X_1,\ldots,X_n\}, \#S=J-1,$$ that is, the knots are sampled randomly without replacement from the observed values of predictor variables (only applicable for deterministic $X$ with distinct values). 
	\item \textbf{Type 3 prior}: The knots are equidistant and the number of steps $J$ is given a prior satisfying 
	\begin{equation} 
	\label{prior J}
	\exp [-b_1j(\log j)^{t_1}]  \leq \Pi(J=j) \leq \exp [-b_2j(\log j)^{t_2}]
	\end{equation} 
	for some $b_1, b_2>0$ and $0\le t_2\leq t_1\leq 1$.
\end{enumerate}

In all three cases, given $\sigma$ and $J$, the coefficients $\theta_1,\ldots,\theta_j$ are given independent normal priors $\theta_j|\sigma \sim \N (\zeta_j, \sigma^2\lambda_j^2)$, $B_1<\lambda_j<B_2$ for some $B_1, B_2>0$ and bounded $|\zeta_1|,\ldots,|\zeta_J|$. We write $\bm \Lambda=\mathrm{diag} (\lambda_1^2,\ldots,\lambda_J^2)$, the diagonal matrix with entries $\lambda_1^2,\ldots,\lambda_J^2$. Hence the prior Type 1 prior will be used to obtain optimal posterior contraction in $\LL_1$-distance, Type 2 prior for posterior contraction in terms of an empirical $\LL_2$-distance while Type 3 prior will be used for testing monotonicity against smooth alternatives of unspecified smoothness. 

The variance parameter $\sigma^2$ is either estimated by maximizing the marginal likelihood, or is given an inverse-gamma prior $\sigma^2\sim \mathrm{IG}(\beta_1,\beta_2)$ with $\beta_1>2$ and $\beta_2>0$. 

We write $\boldsymbol{Y}=(Y_1,\ldots,Y_n)\trans$, $\boldsymbol{X}=(X_1,\ldots,X_n)\trans$, $D_n= (\boldsymbol{Y}, \boldsymbol{X})$, $\boldsymbol{\varepsilon}=(\varepsilon_1,\ldots,\varepsilon_n)\trans$, $\bm B=(\!( \Ind \{X_i\in I_j\} )\! )$, an $n\times J$ matrix, and $\bm \theta=(\theta_1,\ldots,\theta_j)\trans$. Thus the model can be written as $\bm Y=\bm B \bm \theta+\bm \varepsilon$, and the prior (given $J$ and $\sigma$) as $\bm \theta|J \sim \N_J (\bm \zeta, \sigma^2\bm \Lambda)$ with $\bm \zeta=(\zeta_1,\ldots,\zeta_J)\trans$. Then $\bm \theta|(D_n,J,\sigma,\bm \xi)\sim \N_J ((\bm B\trans \bm B+\bm \Lambda)^{-1} (\bm B\trans \bm Y+\bm \Lambda^{-1}\bm \zeta), \sigma^2 (\bm B\trans \bm B+\bm \Lambda^{-1})^{-1})$, that is, $\theta_j$ are a posteriori independent with 
\begin{equation} 
\label{eq:posterior}
\theta_j |(\bm \xi,\sigma,J,D_n) \sim \N \bigg (\frac{{N_j \bar{Y}_j}+{\zeta_j}/{\lambda_j^2}}{{N_j}+{1}/{\lambda_j^2}}, \frac{\sigma^2}{{N_j}+{1}/{\lambda_j^2}}  \bigg ). 
\end{equation}
The marginal distribution of the observations $\bm Y$ (given $\bm X$ and $J,\sigma^2,\bm \xi$) is 
\begin{equation} 
\label{eq:marginal}
\boldsymbol{Y}|(\sigma, \bm \xi, J, \bm X)\sim\mathrm{N}_n\big (\boldsymbol{B\zeta},\sigma^2(\boldsymbol{B\Lambda \bm B}\trans+\boldsymbol{I}_n)\big ). 
\end{equation} 
As the coefficients $\bm \theta$ have not been restricted to the cone of monotone increasing values $\mathcal{Q}:=\{(q_1,\ldots,q_J): q_1\le q_2\le \cdots \le q_J\}$, the resulting regression function $f=\sum_{j=1}^J \theta_j \mathbbm{1}_{I_j}$ may not be monotone. In order to comply with the monotonicity restriction, a sampled value of the function $f$ from its posterior (obtained through the posterior sampling of $\bm \theta$) is projected on the set of monotone functions $\mathcal{F}_+$ on $[0,1]$ to obtain $f^*\in \mathcal{F}_+$ nearest to $f$ with respect to some distance $d$. The induced distribution of $f^*$ will be called the projection-posterior distribution. It will be denoted by $\Pi_n^*$ and will be the basis of inference on the regression function $f$. By its definition, the projection-posterior distribution is restricted to $\mathcal{F}_+$. 

We also find that the projection $f^*$ of a step function $f=\sum_{j=1}^J \theta_j \mathbbm{1}_{I_j}\in \mathcal{F}_J$ is itself a step function $f=\sum_{j=1}^J \theta_j^* \mathbbm{1}_{I_j}\in \mathcal{F}_J$, with $\theta_1^*\le \cdots \le \theta_J^*$. For the the $\LL_2(G_n)$-distance,  these values are obtained by the weighted isotonization procedure 
\begin{equation} 
\label{isotonization}
\mbox{minimize } \sum_{j=1}^J N_j(\theta_j-\theta^*_j)^2 \mbox{ subject to } \theta_1^*\leq\cdots\leq\theta_J^*.
\end{equation}
The optimizing values $\theta_1^*,\ldots,\theta_J^*$ can be computed using the PAVA and can be characterized as the left-derivative at the point $n^{-1} \sum_{k = 1}^{j}N_k$ of the greatest convex minorant of the graph of the line segments connecting the points  
$
\big\{(0,0), \big({N_{1}}/{n}, {N_{1}}\theta_1/n\big), \ldots, \big(\sum_{k = 1}^{J}{N_{k}}/{n}, \sum_{k=1}^{J} {N_{k}}\theta_k/n\big)\big\}$ (cf. Lemma~2.1 of Groeneboom and Jongbloed \cite{Groeneboom2014}). the same solution is obtained even if the $\LL_2(G_n)$-distance is replaced by a wider class; see Theorem~2.1 of Groeneboom and Jongbloed \cite{Groeneboom2014}.

We make one of the following design assumptions (DD) or (DR) on the predictor $X$ and the assumption (E) on the error variables. 

{\bf Condition (DD)} (Deterministic predictor). The predictor variables $X$ is deterministic assuming values $X_1,\ldots,X_n$, and the counts $N_1,\ldots,N_J$ of $J$ equispaced intervals $I_1,\ldots,I_J$ satisfy, for $J\to \infty$,  $\max \{ N_j: 1\le j\le J\}/n\to 0$. 

The bounds are clearly implied by the condition  
$\sup\{ |G_n(x)-G(x)|:\, x\in [0,1]\}=o(J^{-1})$, where $G$ has a positive and continuous density $g$ on $[0,1]$. 

{\bf Condition (DR)} (Random predictor). The predictor $X$ is sampled independently from a distribution $G$, having a density $g$, which is bounded and bounded away from zero on $[0,1]$.

The assumption of normality on the error is only a working hypothesis. We assume the following condition on the error. 

{\bf Condition (E)} (True error distribution). 
The error variables $\varepsilon_1,\ldots,\varepsilon_n$ are i.i.d. sub-Gaussian with mean $0$ and variance $\sigma_0^2$.

We denote the true value of the regression function by $f_0$ and write the vector of function values at the observed points by 
$\boldsymbol{F}_0=(f_0(X_1),\ldots,f_0(X_n))$ and the corresponding true distribution by $P_0$. Let $\mathrm{E}_0(\cdot)$ and $\mathrm{Var}_0(\cdot)$ be the expectation and variance operators taken under the true distribution $P_0$.

The error variance $\sigma^2$ may be estimated by maximizing the marginal likelihood of $\sigma$. From \eqref{eq:marginal}, it follows that the marginal maximum likelihood estimate of $\sigma^2$ is given by 
\begin{equation} 
\label{eq:MLE}
\hat{\sigma}_n^2=n^{-1}{(\boldsymbol{Y}-\boldsymbol{B\zeta})\trans(\boldsymbol{B\Lambda B}\trans+\boldsymbol{I}_n)^{-1}(\boldsymbol{Y}-\boldsymbol{B\zeta})}.
\end{equation}
The plug-in posterior distribution of $f$ is then obtained by substituting $\hat{\sigma}_n$ for $\sigma$ in \eqref{eq:posterior}. 
If instead, we equip $\sigma^2$ with inverse-gamma prior $\sigma^2\sim \mathrm{IG}(\beta_1,\beta_2)$, then a fully Bayes procedure can be based on the posterior distribution 
\begin{equation} 
\label{eq:hierarchical}
\hat{\sigma}_n^2 \sim \mathrm{IG}(\beta_1+n/2, \beta_2+(\boldsymbol{Y}-\boldsymbol{B\zeta})\trans (\boldsymbol{B\Lambda B}\trans+\boldsymbol{I}_n)^{-1}(\boldsymbol{Y}-\boldsymbol{B\zeta})/2).
\end{equation}

\section{Posterior contraction rates under monotonicity}
\label{sec:rates}

\subsection{Preliminaries} 

To establish posterior contraction rates for $f$ with unknown $\sigma$, we need to effectively control the range of values of $\sigma$. 

It will be shown in Lemma~\ref{th:tsigma2con} that the maximum marginal likelihood estimator for $\sigma^2$ in the plug-in Bayes approach or the marginal posterior distribution of $\sigma^2$ in the fully Bayes approach, are consistent for any $f_0\in \mathcal{F}_+$, and the convergence is also uniform over $\mathcal{F}_+(K)$, for any fixed $K>0$. This allows us to treat $\sigma$ as essentially known in studying the posterior contraction.

As mentioned in the last section, we impose monotonicity on $f$ by projecting $f$ on $\mathcal{F}_+$ and use the projection posterior distribution for inference. The following argument shows that the concentration property of the posterior at any monotone function is not weakened by this procedure.  

Let $\Pi_n^*$ stand for the projection posterior distribution given by 
\begin{equation}
\label{projection-posterior} 
\Pi_n^*( B) = \Pi(f: f^* \in B|D_n), \quad B\subset \mathcal{F},
\end{equation}
where $f^*$ is the projection of $f$ on $\mathcal{F}_+$ with respect to some metric $d$ on the space of regression functions. Then for the true regression function $f_0\in \mathcal{F}_+$ and $\epsilon>0$, we have that 
\begin{equation}
\label{projection connection} 
\Pi_n^*(d(f,f_0)>2\epsilon)\le \Pi(f: d(f^*,f_0)>\epsilon|D_n), 
\end{equation}
and hence the contraction rate of the unrestricted posterior is inherited by the projection posterior. To see this, note that 
$d(f^*, f)\leq d(f_0, f)$ by the property of the projection. Hence, using the triangle inequality 
\begin{align}
\label{eq: rate inherited by projection}
&d(f^*, f_0) \leq d(f^*, f) + d(f, f_0) \leq d(f_0, f) + d(f, f_0) = 2d(f, f_0).
\end{align}
For $p\geq 1$, the $\LL_p$-projection of a step function is easily computable, by algorithms similar to the PAVA (see Section 3.1 of De Leeuw et al. \cite{PAVAinR}).

\subsection{Contraction rate under the $\LL_1$-metric} 

In this subsection, we derive the posterior contraction rate with respect to the $\LL_1$-metric. An important factor determining this rate is the approximation rate  of monotone functions by step functions. For the $\LL_1$-metric, step functions with regularly placed knots are adequate for the optimal approximation rate (see Lemma~\ref{approximation}), and hence it is sufficient to consider a Type 1 prior. In the following theorem, we derive the contraction rate at a monotone function in the $\LL_1$-metric by directly bounding posterior moments.

\begin{theorem}
	\label{th: contraction rate monotone}
	Let $f_0\in\mathcal{F}_+$, and assume that Condition {\rm (E)} holds. Let the prior on $f$ be of Type $\mathrm{1}$, with $J\rightarrow\infty$ and $J\ll n$. Let $\sigma^2$ be estimated using the plug-in Bayes approach or endowed with the inverse-gamma prior using a fully Bayes approach. Assume that either $X$ is deterministic and Condition {\rm (DD)} holds, or $X$ is random and Condition {\rm (DR)} holds. Then for $\epsilon_n=\max\{J^{-1}, (J/n)^{1/2}\}$ and every $M_n \rightarrow \infty$, 
	\begin{enumerate}
		\item [{\rm (a)}] $\mathrm{E}_0\ \Pi_n^*\left(\|f-f_0\|_{1,G_n}  > M_n\epsilon_n \right) \rightarrow0$ for the fixed design;  
		\item [{\rm (b)}]  $\mathrm{E}_0\ \Pi_n^*\left(\|f-f_0\|_{1,G} > M_n\epsilon_n \right) \rightarrow0$ for the random design. 
	\end{enumerate}
	In particular, if we choose $J \asymp n^{1/3}$, the projection-posterior contracts at the minimax rate $\epsilon_n = n^{-1/3}$. Moreover, the convergence is uniform over $\mathcal{F}_+(K)$ for any $K>0$. 
\end{theorem}

Under Condition (DR), the $\LL_1(G)$-distance is equivalent to the usual $\LL_1$-metric on $[0,1]$, and hence the contraction rate may be stated in terms of the latter. Conditions (DD) or (DR) on $X$ in the theorem above is needed only to conclude, using Lemma~\ref{th:tsigma2con}, that the estimator (or the posterior) for $\sigma$ is consistent. The conclusion is only used to get an upper bound for $\sigma$. If instead, we assume an upper bound for $\sigma$ (and change the prior on $\sigma$ to comply with the bound, if the fully Bayes procedure is used), we can remove these conditions. 

\subsection{Contraction rates under the empirical $\LL_p$-metric}

When the metric under consideration is $\LL_p$ with $p > 1$, step functions based on equidistant knots do not have the optimal approximation property. To restore this ability, we need to allow arbitrary knots (see Lemma~\ref{approximation}), and put a prior on these. Then the theory of posterior contraction for general (independent, not identically distributed) observations of Ghosal and van der Vaart \cite{noniid} can be applied by computing the prior concentration rate near the truth and bounding the metric entropy of a suitable subset of the parameter space, called a sieve. However, due to their ordering requirement and possibly very uneven allocation of the knots $\bm \xi$ used for the construction of the optimal approximation, the concentration of the prior distribution of $\bm \xi$ near their values appearing in the optimal approximation may be low, and hence the posterior concentration rate may suffer. The problem can be avoided by choosing knots from the observed values of $X$ when the predictor variable is deterministic and the empirical $\LL_p$-norm $\|f\|_{p,G_n}$ is used. Then the optimal rate (up to a logarithmic factor) can be obtained.

\begin{theorem}
	\label{th:l2 contraction}
	Let $X$ be deterministic assuming values $X_1,\ldots,X_n$. 
	Let $f_0\in\mathcal{F}_+$ and the prior on $f$ be of Type $\mathrm{2}$, with $\log J \asymp \log n$. Let $\varepsilon_1,\ldots,\varepsilon_n$ be i.i.d. normal with mean zero and  variance $\sigma^2$, which is estimated using the plug-in Bayes approach or is endowed with the inverse-gamma prior using a fully Bayes approach. Then for any $0<p\le 2$, $\mathrm{E}_0\ \Pi_n^*\left(\|f-f_0\|_{p,G_n}  > M_n\epsilon_n \right) \rightarrow 0$, where 
	$\epsilon _n =  \max \{\sqrt{(J \log n)/n}, J^{-1}\}$. In particular, the best rate $\epsilon_n=(n/\log n)^{-1/3}$ is obtained by choosing $J\asymp (n/\log n)^{1/3}$. Moreover, the conergence is uniform over $\mathcal{F}_+(K)$ for any $K>0$. 
	
	If instead of choosing $J$, we put a prior also on $J$ following \eqref{prior J}, then the contraction rate is given by $n^{-1/3}(\log n)^{(5-3t_2)/6}$. 
\end{theorem}

Clearly,  with a prior on $J$ given by \eqref{prior J}, the best rate $(n/\log n)^{-1/3}$ is obtained when $t_1=t_2=1$. A Poisson (or a suitably truncated Poisson) prior meets the requirement. Again, Condition (DD) is used only to derive the consistency of the estimator (or the posterior) of $\sigma$, and the condition can be removed if $\sigma$ is assumed to be bounded.

It would be interesting to obtain nearly optimal contraction rates for the continuous $\LL_p$-metric, but we do not know an appropriate prior on the knot-locations that would allow sufficient prior concentration to yield the desired result. For a continuous metric $\LL_p$-metric, the weak approximation with equal intervals allows only a sub-optimal approximation rate $J^{-1/p}$ (see Lemma~\ref{approximation}), and consequently a suboptimal posterior contraction rate $(n/\log n)^{-1/(p+2)}$.  

\section{Bayesian testing for monotonicity of $f$}
\label{sec:testing}

A natural test for the hypothesis of monotonicity is given by the posterior probability of $\mathcal{F}_+$: reject the hypothesis if $\Pi (f\in \mathcal{F}_+|D_n)$ is smaller than $1/2$, say. The problem with this test is that if the true regression $f_0\in \mathcal{F}$ belongs to the boundary of $\mathcal{F}_+$, then even if the posterior is consistent at $f_0$, the posterior probability $\Pi (f\in \mathcal{F}_+|D_n)$ may be low because a large part of a neighborhood of $f_0$ may fall outside $\mathcal{F}_+$. In order to avoid such false rejections, one may quantify a test based on a discrepancy measure $d(f,\mathcal{F}_+)$ between $f$ sampled from the posterior, and the set of monotone functions $\mathcal{F}_+$ (that is, a nonnegative function of $f$ that vanishes exactly on $\mathcal{F}_+$), or equivalently, $d(f,f^*)$ where $f^*$ is the projection of $f$ on $\mathcal{F}_+$. A reasonable test can be based on the posterior probability $\Pi (f: d(f,\mathcal{F}_+)<\tau_n|D_n)$ for a sequence $\tau_n\to 0$ slowly. This test is equivalent to rejecting for low values of the posterior probability $\Pi (\mathcal{F}_+^{\tau_n}|D_n)$ of the $\tau_n$-neighborhood $\mathcal{F}_+^{\tau_n}=\{f: d(f,\mathcal{F}_+)<\tau_n\}$  of $\mathcal{F}_+$. This approach was also pursued by Salomond \cite{Salomond2014a,salomond2018}, with a discrepancy measure given by $d(f,\mathcal{F}_+)=\max \{ (\theta_j -\theta_i): 1\le j\le i\le J\}$ for $f=\sum_{j=1}^J \theta_j \Ind_{I_j}$ (with equidistant knots) and a cut-off $\tau_n =\sqrt{(J\log n)/n}$. This test has probability of Type I error going to zero and has high power against smooth alternatives, if appropriately separated from the null. However, the power of this test at a non-smooth alternative may not go to one. This prompts us to propose an alternative test, based on the $\LL_1$-distance as the discrepancy measure, which has the property of universal consistency, that is, the power at any fixed alternative goes to one.   

Let $\mathcal{H}(\alpha,L)$ be the H\"older space of $\alpha$-smooth function with H\"older norm bounded by $L$ (see Definition~C.4 of Ghosal and van der Vaart \cite{ghosal2017fundamentals}). 
\begin{theorem}
	\label{th: testing 1}
	Consider a Type {\rm 1} prior with $J\asymp n^{1/3}$. Let $\sigma^2$ be estimated using the plug-in Bayes approach or endowed with the inverse-gamma prior using a fully Bayes approach. Assume that $X$ is random and Condition {\rm (DR)} holds, and the errors satisfy Condition {\rm (E)}. For $d(f_1,f_2)=\int |f_1-f_2|dG$, consider 
	the test defined by $\phi_n = \Ind \{  \Pi(d(f,\mathcal{F}_+) \leq  M_n  n^{-1/3}|D_n) < \gamma\}$, where  $0<\gamma<1$ is a predetermined constant and $M_n\to \infty$ is  fixed slowly growing sequence. Then the following assertions hold. 
	\begin{enumerate}
		\item [{\rm (a)}] $($Consistency under $H_0):$ For any fixed $f_0 \in \mathcal{F}_+$, $\mathrm{E}_0\phi_n \rightarrow 0$, and further the convergence is uniform over $\mathcal{F}_+(K)$.
		\item [{\rm (b)}] $($Universal Consistency$):$ For any fixed $f_0$ integrable on $[0, 1]$ and $f_0 \notin \bar{\mathcal{F}}_+$, $\mathrm{E}_0(1 - \phi_n) \rightarrow 0$.
		\item [{\rm (c)}] $($High power at converging smooth alternatives$):$ For any $0<\alpha\le 1$ and $L>0$, 
		$
		\sup\{ \mathrm{E}_0 (1 - \phi_n) : f_0 \in \mathcal{H}(\alpha, L), d(f_0, \mathcal{F}_+) > \rho_n(\alpha)\} \rightarrow 0$, 
		where 
		$$\rho_n(\alpha)=\begin{cases} 
		C n^{-\alpha/3}, & \mbox{ for some }C>0 \mbox{ if }\alpha<1,\\ 
		
		C M_n n^{-1/3}, & \mbox{ for any }C>1 \mbox{ if }\alpha=1.  
		\end{cases}$$
	\end{enumerate}
\end{theorem}	

In the above theorem, the $\LL_1(G)$-distance may be replaced by the $\LL_1$-distance under the Lebesgue measure, since under Condition (DR), these two metrics are equivalent. In this case, part (c) may be strengthened by replacing the H\"older space $\mathcal{H}(\alpha, L)$ by the Sobolev space with $(1,\alpha)$-Sobolev norm bounded by $L$ (see Definition~C.6 of Ghosal and van der Vaart \cite{ghosal2017fundamentals}). Also, if $G$ is replaced by the empirical distribution $G_n$ (and assuming that Condition (DD) holds instead of Condition (DR) if $X$ is deterministic), the conclusions in parts (a) and (c) will still hold. The proof is very similar. If $\sigma$ has a known bound, then Condition (DD) or Condition (DR) is not needed. 

The procedure involving the test $\phi_n$ is computationally simple as it does not involve a prior on $J$. The algorithm for median isotonic regression (see Robertson and Wright \cite{robertson1973} and De Leeuw et al. \cite{PAVAinR}) allows us to compute $d(f, \mathcal{F}_+)$ very efficiently. However, with a deterministic choice of $J$, the posterior contraction is not adaptive on classes of functions with different smoothness $\alpha$. Therefore an order of separation $n^{-\alpha/3}$ (up to a logarithmic factor) is needed, which is larger than the optimal order $n^{-\alpha/(1+2\alpha)}$ of separation for $\alpha<1$. Adaptation can however be restored by using a prior on $J$ and letting cut-off value for the discrepancy with $\mathcal{F}_+$ depend on $J$, as in Salomond \cite{salomond2018}, if the class of regression functions is uniformly bounded.

\begin{theorem}
	\label{th: testing 2}
	Let the prior on $f$ be of Type $\mathrm{3}$ with $J$ given a Poisson prior, and $\sigma$ be bounded and be given a positive prior density with bounded support containing the true value $\sigma_0$. Assume that $X\sim G$ and $G$ satisfies Condition {\rm (DR)}. Let $\phi_n= \Ind \{  \Pi(d(f,\mathcal{F}_+) \leq  M_0 \sqrt{(J\log n)/n} |D_n) < \gamma\}$, where $d$ is the Hellinger distance on $p_f(y,x)= (2\pi\sigma^2)^{-1/2}\exp[-(y-f(x))^2/(2\sigma^2)]g(x)$, the density  induced by $f$, $0<\gamma<1$ is a predetermined constant and $M_0>0$ is a sufficiently large constant. 
	\begin{enumerate}
		\item [{\rm (a)}] $($Consistency under $H_0):$ For any fixed $f_0 \in \mathcal{F}_+$, $\mathrm{E}_0\phi_n \rightarrow 0$, and the convergence is uniform over $\mathcal{F}_+(K)$.
		\item [{\rm (b)}] $($Universal Consistency$):$ For any fixed $f_0$ integrable on $[0, 1]$ and $f_0 \notin \bar{\mathcal{F}}_+$, $\mathrm{E}_0(1 - \phi_n) \rightarrow 0$.
		\item [{\rm (c)}] $($Adaptive power at converging smooth alternatives$):$
		For $f_0 \notin \mathcal{F}_+$, $f_0 \in \mathcal{H}(\alpha, L)$, there exists $C$ depending on $\alpha$ and $L$ only such that 
		$$
		\sup\{ \mathrm{E}_0 (1 - \phi_n) : f_0 \in \mathcal{H}(\alpha, L), d(f_0, \mathcal{F}_+) > C (n/\log n)^{-\alpha/(1+2\alpha)} \} \rightarrow 0.$$
	\end{enumerate}
\end{theorem}
In the theorem, $G$ can be replaced by the uniform distribution in the definition of the test. 
In this case, the H\"older space $\mathcal{H}(\alpha, L)$ in part (c) can be replaced by the Sobolev space with $(2,\alpha)$-Sobolev norm bounded by $L$.

Unlike Theorem~\ref{th: testing 1}, the proof requires the application of the general theory of posterior contraction. The weaker Hellinger distance for separation is used so that a test required for the application of the theory is available automatically without requiring the regression functions to be bounded by a constant, a condition that will rule out the conjugate normal prior needed in the proof. An alternative is to use the empirical $\LL_1$-distance and conclude parts (a) and (c) only, assuming that $N_j\asymp n/J$ uniformly in $j=1,\ldots,J$.

\section{Proofs of the main results}
\label{sec:proofs}

\begin{proof}[Proof of Theorem \ref{th: contraction rate monotone}]
	In view of \eqref{projection connection}, it is enough to obtain the contraction rate of the unrestricted posterior. 
	We prove the result for the plug-in Bayes approach; the fully Bayes case can be dealt with similarly. 
	From Lemma~\ref{th:tsigma2con}, get a shrinking neighborhood $\mathcal{U}_n$ of $\sigma_0$ with $P_0(\hat\sigma\in\mathcal{U}_n) \rightarrow 1$. Hence for the purpose of the proof, we may assume that $\hat\sigma\in\mathcal{U}_n$. 
	
	We first consider the case that $X$ is deterministic. 
	Let $f_{0J}=\sum_{j=1}^J\theta_{0j}\Ind_{I_j}$ with $\theta_{0j} =  N_j^{-1} \sum_{i:X_i\in I_j}f_0(X_i)$ for all $1\leq j\leq J$. By Lemma~\ref{approximation} (a), $\|f_{0J}-f_0\|_{1,G_n} \lesssim J^{-1}$ and the bound is also uniform for $f_0\in \mathcal{F}_+(K)$. To complete the proof, we now show that 
	\begin{equation} 
	\label{label0}
	\E_0 \Pi(\|f - f_{0J}\|_{1,G_n}>M_n\sqrt{J/n} \big| D_n)\to 0 \mbox { for any  } M_n\to \infty.
	\end{equation}  
	
	Since $f=\theta_j$ and $f_0=\theta_{0j}$ on $I_j$, $\|f - f_{0J}\|_{1,G_n}=n^{-1}\sum_{j=1}^J {N_j} |\theta_j-\theta_{0j}|$. 
	Hence by the Cauchy-Schwarz inequality followed by Markov's inequality,
	\begin{equation}
	\label{label1}
	\Pi(\|f - f_{0J}\|_{1,G_n}>M_n\sqrt{J/n} \big| D_n)\lesssim  \frac{1}{M_n^2 J}\sum_{j=1}^J N_j \mathrm{E}(|\theta_j-\theta_{0j}|^2\big| D_n).
	\end{equation}
	For $1\leq j \leq J$, we bound $\mathrm{E}(|\theta_j-\theta_{0j}|^2\big| D_n)= \mathrm{Var}(\theta_j\big| D_n) + |\mathrm{E}(\theta_j\big| D_n) - \theta_{0j}|^2$, bound the expectation of both terms, and put in \eqref{label1} to obtain the desired result. 
	For the first term,  
	\begin{align}
	\label{eq: contraction main calculation 2}
	N_j \mathrm{Var}(\theta_j\big| D_n) \le  \sup_{\sigma\in\mathcal{U}_n}\frac{N_j \sigma^2}{[N_j+\lambda_j^{-2}]^{1/2}} 
	\lesssim 1.
	\end{align}
	We bound $\mathrm{E}_0[ N_j |\mathrm{E}(\theta_j\big| D_n) - \theta_{0j}|^2]$ as 
	\begin{align*}
	\mathrm{E}_0 \bigg [ N_j \bigg| \frac{N_{j}\bar{Y_{j}} + \frac{\zeta_{j}}{\lambda_{j}^{2}}}{N_{j} + \frac{1}{\lambda_{j}^{{2}}}} - \frac{\sum_{i:X_i\in I_j}f_0(X_i)}{N_j}\bigg |^2 \bigg ] & \lesssim 1+ \mathrm{E}_0\bigg| \frac{\sum_{i:X_i\in I_j}(Y_i-f_0(X_i))}{N_j}\bigg | . 
	\end{align*}
	Using the boundedness of $\zeta_j$ and $\lambda_j^{-2}$, and the second term in the last expression is bounded by $\sigma_0^2$ by the moment inequality. 
	
	For random predictors, we use the $\|\cdot\|_{1,G}$-distance, which involves another integration with respect to $X_1,\ldots,X_n$ on the left side of \eqref{label0}. 
	
\end{proof}

\begin{proof}[Proof of Theorem \ref{th:l2 contraction}]
	Because of \eqref{projection connection}, it suffices to obtain the contraction rate of the unrestricted posterior. 
	Since for $0<p<2$, the $\LL_p(G_n)$-distance is dominated by the $\LL_2(G_n)$-distance, it suffices to prove the result for $p=2$.  
	We shall apply the general theory of posterior contraction (Ghosal and van der Vaart \cite{ghosal2017fundamentals}, Chapter 8) using the sieve 
	\begin{equation} 
	\label{sieve1}
	\mathcal{P}_{n} = \big\{f  = \sum_{j=1}^{J} \theta_{j}\Ind_{[\xi_{j-1}, \xi_{j})}, \ \xi_{ 1}, \ldots, \xi_{J-1} \in \boldsymbol{X}, \max_j |\theta_j| \leq n \big\}.  
	\end{equation} 
	Let $p_{f,\sigma}^{(n)}$ denote the joint density of $Y_1,\ldots,Y_n$ for a regression function $f$. 
	We verify the conditions of Theorem~8.26 of Ghosal and van der Vaart \cite{ghosal2017fundamentals} for $\epsilon_n=\max\{ \sqrt{(J \log n)/n},J^{-1}\}$. Note that by Lemma~\ref{th:tsigma2con}, we can restrict $\sigma$ to  an arbitrarily small neighborhood of $\sigma_0$, so the test construction in Lemma~8.27 of Ghosal and van der Vaart \cite{ghosal2017fundamentals} is applicable.
	
	By direct calculations, the Kullback-Leibler divergence and the square Kullback-Leibler variation  are respectively equal to
	\begin{align*}
	K(p_{f_0,\sigma_0}^{(n)}; p_{f,\sigma}^{(n)})= \E_0 \log \frac{p_{f_0,\sigma_0}^{(n)}}{p_{f,\sigma}^{(n)}}=\frac{ n } {2\sigma^2} \|f-f_0\|^2_{2, G_n}+ \frac{n}{2}\big [\frac{\sigma_0^2}{\sigma^2}-1-  \log \frac{\sigma_0^2}{\sigma^2}\big], \\
	V_{2,0}(p_{f_0,\sigma_0}^{(n)}; p_{f,\sigma}^{(n)})= \mathrm{Var}_0	\log \frac{p_{f_0,\sigma_0}^{(n)}}{p_{f,\sigma}^{(n)}}= \frac{n}{4} \big ( \frac{\sigma_0^2}{\sigma^2}-1\big )^2+ \frac{n\sigma_0^2}{\sigma^4} \|f-f_0\|^2_{2, G_n}.
	\end{align*}
	Therefore for a sufficiently small $\epsilon$, there exists $C_1>0$ such that
	\begin{align*}
	B_{n, 0}((f_0, \sigma_0), \epsilon) &:= \{(f, \sigma): K(p_{f_0,\sigma_0}^{(n)}, p_{f,\sigma}^{(n)})\leq n\epsilon^2, V_{2,0}(p_{f_0,\sigma_0}^{(n)}; p_{f,\sigma}^{(n)}) \leq n\epsilon^2\} \\
	& \supset \{(f, \sigma): \|f- f_0\|^2_{2, G_n}\leq C_1\epsilon^2, |\sigma^2-\sigma_0^2|^2\leq C_1\epsilon^2\}.
	\end{align*} 
	By Lemma~\ref{approximation}, there exists $f_{0J}$ such that  $f_{0J}(\cdot)=\sum_{j=1}^J\theta_{0j}\Ind_{I_j}$, where $I_1,\ldots,I_J$ are an interval partition with knots $\{\xi_{0,1},\ldots, \xi_{0, J-1}\}\subset \{X_1,\ldots,X_n\}$ and $\|f_{0J}- f_0\|_{2, G_n}^2 \lesssim \epsilon_n^2$. By the prior independence of $f$ and $\sigma$, and because $-\log \Pi (|\sigma-\sigma_0|^2\le C\epsilon_n^2)\lesssim \log (1/\epsilon_n)\lesssim \log n$, it suffices that 
	\begin{align*}
	\Pi (\|f-f_{0J}\|_{2, G_n}^2 \le C_2 \epsilon_n^2)&= \Pi\big(\sum_{i=1}^{n} N_j (\theta_j - \theta_{0j})^2 \leq C_2 n\epsilon_n^2 \big|\boldsymbol{\xi}=\boldsymbol{\xi}_0\big) 
	\Pi(\boldsymbol{\xi}=\boldsymbol{\xi}_0) \\
	& \geq  \Pi\big(\bigcap\limits_{j=1}^{J}\big\{|\theta_j - \theta_{0, j}| \leq \sqrt{C_2} \epsilon_n\big\}\big) \frac{1}{\binom{n}{J-1}},
	\end{align*}
	since $\sum_{i=1}^{n}(f(X_i) - f_{0J}(X_i))^2=\sum_{j=1}^J N_j |\theta_j -\theta_{0j}|^2$ and $\sum_{j=1}^J N_j=n$. The last expression is at least of the order $ (C_3\epsilon_n)^{J}\ n^{-(J-1)}$ for some $C_3>0$. Putting these together, we have $-\log \Pi (B_{n, 0}((f_0, \sigma_0), \epsilon_n))\lesssim J [\log (1/\epsilon_n)+\log J]\lesssim J \log n\lesssim n\epsilon_n^2$ by the definition of $\epsilon_n$, fullfilling the condition of prior probability concentration needed for posterior contraction rate $\epsilon_n$. 
	
	Observe that the metric entropy $\log \mathcal{N}(\epsilon, \mathcal{P}_{n}, \|\cdot\|_{p,G_n})$   
	of the sieve  $\mathcal{P}_{n} $ in \eqref{sieve1} is bounded above by 
	$ J \log(n/\epsilon_n)\lesssim J \log n\lesssim n\epsilon_n^2$. 
	Finally, the prior probability $\Pi (\mathcal{P}_n^c)$ of the complement of the sieve $\mathcal{P}_{n}$ is bounded by 
	$J e^{-n^2/2}\ll e^{-c n\epsilon_n^2}$ for any $c>0$, 
	establishing condition (8.33)  of Ghosal and van der Vaart \cite{ghosal2017fundamentals}. This establishes the rate $\epsilon_n=\max(\sqrt{(J\log n)/n}, J^{-1})$ when $J$ is chosen deterministically. Clearly, the best choice is $J\asymp (n/\log n)^{1/3}$, giving the nearly optimal rate $(n/\log n)^{-1/3}$. 
	
	When $J$ is given a prior, to lower bound $\Pi (B_{n, 0}((f_0, \sigma_0), \epsilon))$, we intersect the set with $\{ J=J_0\}$, where $J_0\asymp (n/\log n)^{1/3}$. This gives an additional factor $e^{-b_1 J_0 (\log J_0)^{t_1}}$, which is absorbed in $e^{-c n \bar \epsilon_n^2}$ by adjusting the constant for a pre-rate $\bar \epsilon_n = (n/\log n)^{-1/3}$, because $t_1\le 1$. Modify the sieve in \eqref{sieve1} by intersecting with $\{J\le J_1\}$, where $J_1$ to be determined. The prior probability of the complement $\mathcal{P}_n^c$ then contributes an extra factor a constant multiple of $e^{-b_2 J_1 (\log J_1)^{t_2} }$ to $J_1 e^{-n^2/2}$. To obtain the final rate, we need to choose $J_1$ such that $J_1 (\log n)^{t_2}$ exceeds a sufficiently large multiple of $n \bar \epsilon_n^2$, and then the rate is given by $\sqrt{(J_1 \log n)/n}=n^{-1/3}(\log n)^{(5-3t_2)/6}$. 
\end{proof}

\begin{proof}[Proof of Theorem \ref{th: testing 1}]
	
	(a) Let $f_0\in\mathcal{F}_+$. Using the definition of projection, 
	\begin{align*}
	&\mathrm{E}_0\Pi(\|f- f^*\|_{1,G}>M_n n^{-1/3}|D_n) \leq \mathrm{E}_0\Pi(\|f- f_0\|_{1,G}>M_n n^{-1/3}|D_n)\to 0 
	\end{align*}
	for $J\asymp n^{1/3}$
	by Theorem \ref{th: contraction rate monotone}. Then it follows that $\E_0 \phi_n=\mathrm{P}_0(\Pi(d(f,\mathcal{F}_+) \le M_n n^{-1/3}|D_n)<\gamma) \to 0$.  Further, the convergence is uniform over $f_0\in\mathcal{F}_+(K)$ for any $K>0$. 
	
	(b) Let $f_0\notin\bar{\mathcal{F}}_+$ be fixed and integrable. Using the properties of the projection, $d(f_0, \mathcal{F}_+) = \|f_0 - f_0^*\|_{1,G}$ is bounded by $\|f_0 - f^*\|_{1,G}$,  which, by the triangle inequality, is further bounded above by
	$$ \|f_0 - f\|_{1,G}+\|f- f^*\|_{1,G} = \|f- f_0\|_{1,G} +d(f, \mathcal{F}_+).$$ This leads to $d(f, \mathcal{F}_+) \geq d(f_0, \mathcal{F}_+)-\|f- f_0\|_{1,G}$, and hence $\Pi(d(f, \mathcal{F}_+) \leq M_n n^{-1/3}\big | D_n) \leq \Pi(\|f_0- f\|_{1,G} + M_n n^{-1/3} \geq d(f_0, \mathcal{F}_+)\big | D_n).$
	
	Let $\theta_{0j}=\int_{I_j}f_0 dG/G(I_j)$, $1\leq j\leq J$. Then as shown in the proof of Theorem \ref{th: contraction rate monotone}, $\Pi(\|f- f_{0J}\|_{1,G} > M_n\sqrt{J/n}\big | D_n)\rightarrow_{P_0}0$, and hence for $J\asymp n^{1/3}$, we have $
	\Pi(\|f- f_{0J}\|_{1,G} > M_n n^{-1/3}\big | D_n)\rightarrow_{P_0}0$. Next, since $f_0$ is integrable, by the martingale convergence theorem, $\|f_0-f_{0J}\|_{1,G}\rightarrow 0$. hence 
	\begin{align*}
	&\lefteqn{\mathrm{E}_0\Pi(\|f- f_0\|_{1,G}+M_n n^{-1/3}\geq d(f_0, \mathcal{F}_+)|D_n)} \\
	& \leq \mathrm{E}_0\Pi\left(\|f- f_{0J}\|_{1,G} \geq d(f_0, \mathcal{F}_+)-\|f_{0J}- f_0\|_{1,G}- M_n n^{-1/3} \big |D_n\right) \to 0
	\end{align*}
	because $d(f_0, \mathcal{F}_+)$ is fixed and positive. This implies that the probability of Type 2 error $\mathrm{P}_0(\Pi (d(f, \mathcal{F}_+) \leq M_n n^{-1/3}|D_n)\ge \gamma)\to 0 $. 
	
	(c) Let $f_0\notin\mathcal{F}_+$ and $f_0\in\mathcal{H}(\alpha,L)$ such that $d(f_0,\mathcal{F})\ge \rho_n(\alpha)$. Consider the step function $f_{0J}$ of $f_0$ as in part (b). By a well-known fact from approximation theory, we have that $\|f_0-f_{0J}\|_{1,G}\le C(L) J^{-\alpha}$ for some constant $C(L)$ depending only on $L$. For instance, the bound follows from de Boor \cite{deBoor} as step functions with equidistant points are B-splines of order $1$. Hence for $J\asymp n^{1/3}$, by we have $\Pi(\|f- f_0\|_{1,G}>M_n n^{-1/3}+ C(L) n^{-\alpha/3}|D_n)\rightarrow_{P_0}0$, uniformly for all $f_0\in\mathcal{H}(\alpha, L)$. Thus $d(f,\mathcal{F}_+)$ is 
	$$d(f,f^*)\ge d(f_0,f^*)-d(f,f_0)\ge d(f_0,\mathcal{F}_+)-d(f,f_0)\ge \rho_n(\alpha)-d(f,f_0),$$ 
	so that $$\Pi (d(f,\mathcal{F}_+)\le M_n n^{-1/3}|D_n)\le \Pi (\|f-f_0\|_{1,G}\ge \rho_n(\alpha)-M_n n^{-1/3}|D_n)	\rightarrow_{P_0}0$$  because for $\alpha<1$, 
	$$\rho_n(\alpha)-M_n n^{-1/3}\ge M_n n^{-1/3}+ C(L) n^{-\alpha/3}$$ for $C>C(L)$, while for $\alpha=1$, 
	$$\rho_n(\alpha)-M_n n^{-1/3}\ge M_n n^{-1/3}+ C(L) n^{-\alpha/3}$$ for $C>1$; the last follows because $M_n\to\infty$. 
\end{proof}

\begin{proof}[Proof of Theorem \ref{th: testing 2}]
	Let $f_0$ be a bounded, measurable true regression function (irrespective of monotonicity or smoothness). 
	For a given $J$, consider $f_{0,J}=\sum_{j=1}^J \theta_{0j}\Ind_{I_j}$ with $\theta_{0j}=\int_{I_j} f_0 dG$, $j=1,\ldots,J$. First, we show that for a given $\gamma'>0$ and sufficiently large $M_0$, 
	\begin{equation}
	\label{label4}
	\E_0 \Pi (\|f-f_{0J}\|_{2,G}\ge M_0 \sqrt{(J \log n)/n}, J\le J_n|D_n)<\gamma',
	\end{equation}
	provided that $\log J_n\asymp \log n$. We write the expression inside the expectation as 
	\begin{equation}
	\label{label6}
	\sum_{J=1}^{J_n} \Pi(J|D_n)\Pi \big( \sum_{j=1}^J  (\theta_j-\theta_{0j})^2 G(I_j)  \ge M_0^2 J (\log n)/n \big|D_n\big),
	\end{equation}
	and bound 
	\begin{align}
	\lefteqn{\Pi \big(\sum_{j=1}^J (\theta_j-\theta_{0j})^2 G(I_j) \ge M_0^2 J (\log n)/n\big|D_n\big)}\nonumber\\
	&\le \frac{n \sum_{j=1}^J G(I_j) [\mathrm{Var}(\theta_j|D_n)+ (\E(\theta|D_n)-\theta_{0j})^2 ]}{M_0^2 J \log n}.
	\label{label5}
	\end{align}
	In view of Condition (DR), $G(I_j)$ are  of the order $1/J$, and by Lemma~\ref{lem:order counts}, $N_j$ are of the order $n/J$ in probability uniformly in $j=1,\ldots,J$. 
	Under the boundedness assumption on the prior parameters and the sampling variance, $ \mathrm{Var}(\theta_j|D_n)\lesssim 1/N_j\lesssim J/n$ with high probability, from the standard expressions for normal-normal conjugate setting (see the proof of Theorem~\ref{th: contraction rate monotone}). 
	
	To estimate $(\E(\theta|D_n)-\theta_{0j})^2$, with $\bar Y_j$ standing for $N_j^{-1} \sum_{i: X_i\in I_j} Y_i$ and  $\bar \varepsilon_j$ standing for $N_j^{-1} \sum_{i: \varepsilon_i\in I_j} Y_i$, we first observe that $|\bar \varepsilon_j|^2\le N_j^{-1} \log n\lesssim (J \log n)/n$ with high probability. Here we have used the maximal norm estimate using the squared-exponential Orlicz norm (see Lemma~2.2.2 of van der Vaart and Wellner \cite{empirical}) and $\#\{\bar \varepsilon_j: j\le J\le J_n\}\lesssim J_n^2$. By the same argument and the boundedness of $f_0$, we also have $$|N_j^{-1}\sum_{i: X_i \in I_j} f(X_i)-\theta_{0j}|^2\lesssim N_j^{-1} \log n\lesssim (J\log n)/n$$ with high probability. Also, $|\bar Y_j|$ is uniformly bounded with high probability, because $Y_i=f_0(X_i)+\varepsilon_i$. Putting in the expression for $\E(\theta|D_n)$, we conclude that $(\E(\theta|D_n)-\theta_{0j})^2\le (J\log n)/n$. 
	
	Putting these estimates in \eqref{label5}, we find that the expression is bounded by $M_0^{-2}$ with high probability simultaneously for all $J\le J_n$. Hence by \eqref{label6}, it follows that \eqref{label4} holds.  
	
	We also observe that, if the posterior contracts at the rate $\epsilon_n$ at $f_0$ in the sense that $\E_0 \Pi (J: d (f,f_0) >M_0 \epsilon_n|D_n)\to 0$ for some $M_0>0$, then 
	\begin{equation}
	\label{label7}
	\E_0 \Pi (J: d (f_{0J},f_0) >M_0 \epsilon_n|D_n)\to 0.		
	\end{equation}
	This follows because $f_{0J}$ is the closest to $f_0$ in $\mathcal{F}_J$, so if for a $J_0$, $d (f_{0J_0},f_0) >M_0 \epsilon_n$, then $\Pi (J=J_0|D_n)\le \Pi (J: d (f_{0J},f_0) >M_0 \epsilon_n|D_n)$. 
	
	(a) 
	If $f_0\in \mathcal{F}_+$, then $f_{0J}\in \mathcal{F}_+$. 
	By Lemma~\ref{approximation}, the $\LL_2$-approximation rate of $\mathcal{F}_J$ with equidistant intervals at a monotone function is $J^{-1/2}$. Then standard arguments as in the proof of Theorem~\ref{th:l2 contraction} show that the prior probability of a Kullback-Leibler neighborhood of size $\epsilon^2$ is bounded below by $\exp\{-C_1 \epsilon^{-2}\log (1/\epsilon)\}$. The required test with respect to $d$ is automatically available, while the sieve can be chosen as in Theorem~\ref{th:l2 contraction} and its entropy can be bounded in the same way by noting that $d$ is bounded by the $\LL_2(G)$-metric, leading to a (suboptimal) contraction rate $\epsilon_n=(n/\log n)^{-1/4}$. It also follows that for $J_n$ a large constant multiple of $\epsilon_n^{-1}$, the prior probability of $J>J_n$ is exponentially small compared with the prior concentration, and hence $\{J>J_n\}$ has a small posterior probability. Since $\log J_n \lesssim \log n$, it follows that \eqref{label4} holds.

	(b) Let $f_0\notin\bar {\mathcal{F}}_+$ be fixed and bounded. 
	By the martingale convergence theorem, $\|f_{0J}-f_0\|_{2,G}\to 0$ as $J\to \infty$, so for a given $\epsilon>0$, we can get $J_0$ (depending on $\epsilon$ but not depending on $n$) such that $\|f_{0J_0}-f_0\|_{2,G}<\epsilon/2$. Then for some $\delta>0$, we have $$\Pi (\|f-f_0\|_{2,G} <\epsilon)\ge \Pi(J=J_0) \Pi (\max \{|\theta_j-\theta_{0j}|: 1\le j\le J_0\} <\delta)>0.$$ Further, for $J_1$ an arbitrarily small multiple of $n/\log n$,  the excess prior probability $\Pi (J>J_1)$ can be bounded by $e^{-bn}$ for some $b>0$ depending on $c$. Considering a sieve 
	$	\mathcal{P}_{n} = \big\{f  = \sum_{j=1}^{J} \theta_{j}\Ind_{I_j}, \max_j |\theta_j| \leq n, J\le J_1 \big\}$, 
	standard estimates gives a bound for its metric entropy an arbitrarily small multiple of $n$. 	 
	Therefore it follows that (see Theorem~6.17 of Ghosal and van der Vaart \cite{ghosal2017fundamentals}) that $\E_0 \Pi(J>J_1|D_n)\to 0$ and the posterior is consistent at $f_0$ with respect to $d$, because $d(f_1,f_2)\le \|f_1-f_2\|_{2,G}$.  
	
	Observe that for any $f\in \mathcal{F}_J$, 
	\begin{equation} 
	\label{label8}
	d(f,\mathcal{F}_+)=d(f,f^*)\ge d(f_0,f_0^*)-d(f,f_{0J})-d(f_{0J},f_0).
	\end{equation}
	Since $f_0\notin\bar {\mathcal{F}}_+$, the first term is a fixed positive number. The second term is bounded by $\sqrt{(J \log n)/n}$ with high posterior probability, and $J$ can be restricted to be at most $J_1$, which can be taken to be an arbitrarily small multiple of $n/\log n$. Hence we can make the second terms as small as we like, with high posterior probability. By \eqref{label7} and posterior consistency, the third term can also be made arbitrarily small with high posterior probability. This shows that $d(f,\mathcal{F}_+)$ larger than some fixed positive number with high posterior probability, and hence it will exceed $\sqrt{(J \log n)/n}$ with high posterior probability for all $J\le J_1$, prompting the test to reject the null hypothesis of monotonicity with true probability tending to one.

	(c) Let $f_0\notin\mathcal{F}_+$ and $f_0\in\mathcal{H}(\alpha,L)$ such that $d(f_0,\mathcal{F}_+)\ge \rho_n(\alpha)$. The proof is very similar to part (b) with the following changes. First, by the well-known approximation rate $J^{-\alpha}$ at  functions  in  $\mathcal{H}(\alpha, L)$ by step functions, and standard arguments as used in part (a) and (b), giving prior concentration and metric entropy bounds, the posterior contraction rate at $f_0$ with respect to $d$ is $\epsilon_n=(n/\log n)^{-\alpha/(2\alpha+1)}$. Also, with high posterior probability, $J$ can be restricted to less than $J_1\asymp n \epsilon_n^2/\log n = (n/\log n)^{1/(2\alpha+1)}$. This bounds the second term by a multiple of $(n/\log n)^{-\alpha/(2\alpha+1)}$ with high posterior probability. Finally, by \eqref{label7}, the third term is also bounded by a multiple of $(n/\log n)^{-\alpha/(2\alpha+1)}$ with high posterior probability.
	Therefore, the expression on the right side of \eqref{label8} is larger than $M_0 \sqrt{(J \log n)/n}$ with high posterior probability. Thus the test rejects the null hypothesis of monotonicity with true probability tending to one. 
\end{proof}

\section{Auxiliary results}
\label{sec:appendix}

\begin{lemma}
	\label{lem:order counts}
	If the predictors are random, Condition {\rm (DR)} holds and $n/J\gg \log J$, then for 	
	$ A_n = \big\{ a_1 n/(2J) \leq \min(N_1,\ldots,N_J) \leq \max (N_1,\ldots,N_J) \leq 2a_2 n/J\big\}$,  
	we have $P_0(A_n)\rightarrow 1$.  
	In other words, $N_1,\ldots,N_J$ are simultaneously of the order $n/J$ in probability.
\end{lemma}

\begin{proof}
	From  $N_j\sim\mathrm{Bin}(n;G(I_j))$ and $a_1/J\le G(I_j)\le a_2/J$ for every $1\leq j\leq J$, a standard large deviation estimate for $\mathrm{P}(N_j \geq 2 a_2n/J)$ is $2 e^{-Cn/J}$ for some constant $C>0$, and similarly for $\mathrm{P}(N_j \leq  a_1n/(2J))$. Adding these probabilities $J$ times,  we get the desired result because the factor $\log J$ can be absorbed in $n/J$. 
\end{proof}

\begin{lemma}
	\label{th:tsigma2con}
	Let the predictors be deterministic satisfying Condition {\rm (DD)} or be random satisfying Condition {\rm (DR)}. Let $f_0 \in \mathcal{F}_+$, the prior on $f$ of Type $\mathrm{1}$, and  Condition {\rm (E)} holds. Then for $J\rightarrow\infty$ such that $J\ll n$, we have
	\begin{enumerate}
		\item [(a)] the maximum marginal likelihood estimator $\hat{\sigma}_n^2$ converges in probability to $\sigma_0^2$ at the rate $\max\{n^{-1/2}, n^{-1}J\}$.
		\item [(b)] If $\sigma^2 \sim \mathrm{IG}(\beta_1,\beta_2)$ with $\beta_1>2$, $\beta_2>0$, then the marginal posterior distribution of $\sigma^2$ contracts at the rate $\max\{n^{-1/2}, n^{-1}J\}$.
	\end{enumerate}
\end{lemma}

\begin{proof}
	(a) Let $f_0\in\mathcal{F}_+$. We first show that there exists $\boldsymbol{\theta}_{0J} = (\theta_{01}, \ldots,\theta_{0J})$ such that  $n^{-1}\|\boldsymbol{F}_0 - \boldsymbol{B}\boldsymbol{\theta}_{0J}\|^2 \lesssim J^{-1}$ for deterministic $\boldsymbol{X}$,  and $n^{-1}\mathrm{E}_G\|\boldsymbol{F}_0 - \boldsymbol{B}\boldsymbol{\theta}_{0J}\|^2 \lesssim J^{-1}$ for random $\boldsymbol{X}$.
	
	On a set with $\min\{N_j: 1\leq j\leq J\} >0$, let $\theta_{0j}=N_j^{-1}\sum_{i:X_i\in I_j}f_0(X_i)$. Using the monotonicity of $f_0$, we write $n^{-1}\|\boldsymbol{F}_0 - \boldsymbol{B}\boldsymbol{\theta}_{0J}\|^2$ as 
	\begin{align}
	\frac1n\sum_{j=1}^J\sum_{i:X_i\in I_j}(f_0(X_i)-\theta_{0j})^2 
	\leq \frac1n\sum_{j=1}^J\sum_{i:X_i\in I_j}(f_0(j/J)-f_0((j-1)/J))^2 \nonumber \\
	= \sum_{j=1}^J \frac{N_j}{n}\left(f_0(j/J)-f_0((j-1)/J)\right)^2. \label{eq: theta0eq1}
	\end{align}
	For deterministic ${X}$, by Condition (DD) and the monotonicity of $f_0$, \eqref{eq: theta0eq1} is bounded by 
	\begin{equation} 
	\label{label9}
	\max_{1\le j\le J}\frac{N_j}{n} \sum_{j=1}^J [f_0(j/J)-f_0((j-1)/J)]^2 
	\le \max_{1\le j\le J}\frac{N_j}{n}(f_0(1)-f_0(0))^2\to 0.
	\end{equation}
	For random ${X}$, using the fact that $N_j\sim \mathrm{Bin}(n;G(I_j))$, the expectation of \eqref{eq: theta0eq1} under $G$ equals to 
	$\sum_{j=1}^JG(I_j)\left(f_0(j/J)-f_0((j-1)/J)\right)^2$, which, in view of  Condition (DR), has the bound $\max_{1\le j\le J} G(I_j) (f_0(1)-f_0(0))^2\to 0$. 
	
	For the rest of the proof, we assume that ${X}$ is fixed, satisfying Condition (DD); the random case can be dealt with similarly, by taking expectation with respect to $G$ and using Condition (DR). We imitate the proof of Proposition 4.1 (a) of Yoo and Ghosal \cite{williamSupNormContraction} but assuming that $f_0$ is monotone instead of smooth. Define $\boldsymbol{U}=(\boldsymbol{B\Lambda B}^T+\boldsymbol{I}_n)^{-1}$. We write $$|\mathrm{E}_0(\hat{\sigma}_n^2)-\sigma_0^2|=|n^{-1}\sigma_0^2\mathrm{tr}(\boldsymbol{U})-\sigma_0^2|+n^{-1}(\boldsymbol{F}_0-\boldsymbol{B\zeta})^T
	\boldsymbol{U}(\boldsymbol{F}_0-\boldsymbol{B\zeta})$$ and bound it by a constant multiple of 
	\begin{align} 
	\lefteqn{n^{-1}[\mathrm{tr}(\boldsymbol{I}_n-\boldsymbol{U})+(\boldsymbol{F}_0-\boldsymbol{B\theta}_{0J})^T
		\boldsymbol{U}(\boldsymbol{F}_0-\boldsymbol{B\theta}_{0J})} \nonumber \\ 
	&& +(\boldsymbol{B\theta}_{0J}-\boldsymbol{B\zeta})^T
	\boldsymbol{U}(\boldsymbol{B\theta}_{0J}-\boldsymbol{B\zeta})]
	\label{eq:tsigma0bias}.
	\end{align}
	Among these terms, only the middle term arising out of the approximation of the true function by step functions, is different --- the other two terms are bounded  by $J/n$ considering step functions as B-splines of order 1 in one dimension. The second term can also be bounded by a multiple of $J^{-1}$ in the same way Yoo and Ghosal \cite{williamSupNormContraction} using the $\LL_2$-approximation rate $J^{-1/2}$ for monotone function, leading the upper bound a multiple of $J/n+J^{-1}$ for the expression in \eqref{eq:tsigma0bias}. 
	
	To complete the proof of part (a), we bound $\mathrm{Var}_0(\hat{\sigma}_n^2)$ by a multiple of $n^{-1}$. Again, we can follow the same steps in the proof of Proposition 4.1 (a) of Yoo and Ghosal \cite{williamSupNormContraction} with the approximate rate for a smooth function replaced by the approximation rate $n^{-1}$ for a monotone function. We also observe that the bounds obtained in the proof are uniform over $f_0\in \mathcal{F}_+(K)$ for any $K>0$. 
	
	Given part (a), the proof of part (b) follows exactly as in the proof of Proposition 4.1 (a) of Yoo and Ghosal \cite{williamSupNormContraction}. 
\end{proof}

\begin{lemma}
	\label{approximation}
	Let $p\ge 1$ and $K>0$. Then for every $f\in\mathcal{F}_+(K)$ and $J>1$, there exist $ \theta_1\le \cdots\le \theta_J$ from $[-K,K]$ such that the following assertions hold. 
	\begin{enumerate} 
		\item [(a)] For any partition intervals $I_1,\ldots,I_J$ and probability measure $H$ satisfying $H(I_j)\le M/J$, with $f_{J}=\sum_{j=1}^J \theta_j \Ind_{I_j} \in\mathcal{F}_+(K)$ we have that $\int |f_0-f_{0J}|^p dH \le MK^p/J$.
		\item [(b)] For any probability measure $H$ and $1\le p<\infty$, there exist knots $0=\xi_{ 0}<\xi_{ 1}<\cdots < \xi_{J-1}<\xi_{J}=1$ from the topological support of $H$ such that for any $f \in\mathcal{F}_+(K)$, the exits a function of the form $f_{J}=\sum_{j=1}^J \theta_j \Ind_{I_j} \in\mathcal{F}_+(K)$ satisfying $\int |f_0-f_{0J}|^p dH \le K^p/J^p$, where  $I_j=[\xi_{j-1},\xi_j)$, $j=1,\ldots,J-1$, $I_J=[\xi_{J-1},\xi_J]$. 
	\end{enumerate}
\end{lemma} 

\begin{proof}
	We bound the discrepancy $\int |f-f_J|^p dH= \sum_{j=1}^J \int_{I_j} |f-f_J|^p dH$ by 
	\begin{align*}
	\sum_{j=1}^J H(I_j) |f(j/J)-f((j-1)/J)|^p 
	\le  MJ^{-1}   \sum_{j=1}^J  |f(j/J)-f((j-1)/J)|^p,
	\end{align*} 
	which is bounded by 
	$ |f(1)-f(0)|^p$ 
	by the estimate $\sum a_k^p\le (\sum a_k)^p$ for positive numbers $a_1,\ldots,a_k$ and $p\ge 1$. 
	
	The proof of part (b) is essentially contained in the proof of Theorem 2.7.5 of van der Vaart and Wellner  \cite{empirical}, although their theorem is about a bound for the bracketing or metric entropy. Implicit in their construction is that, given $\epsilon>0$, there exists a $J=J(\epsilon)\lesssim \epsilon^{-1}$, $0\le \xi_1<\cdots<\xi_{J-1}\le 1$ and $\theta_1,\ldots,\theta_J$ such that $f_J=\sum_{j=1}^J \theta_j \Ind_{I_j}$ satisfies $\|f-f_J\|_{p,H}<\epsilon$, where $I_1,\ldots,I_J$ form an interval partition of $[0,1]$ with knots $0=\xi_0< \xi_1<\cdots<\xi_{J-1}\le \xi_J=1$. For instance, one of the lower brackets in their construction of an $\epsilon$-bracketing will satisfy the approximation property. The role of $\epsilon$ and $J$ can be reversed, in that, given $J$, we can first obtain $\epsilon>0$ such that the corresponding $J(\epsilon)$ is within $J$. 
	
	Finally, we need to conclude that the knot points $\xi_1<\cdots<\xi_{J-1}$ can be chosen from the support of $H$. The construction in van der Vaart and Wellner  \cite{empirical} assumed, without loss of generality, that $H$ is uniform. For a general $H$, the quantile transform is applied, transforming the $j$th knot $\xi_j$ to $H^{-1}(\xi_j)$, which belongs to the support of $H$. 
\end{proof}

\bibliographystyle{imsart-number}
\bibliography{mybibfile}


\end{document}